\documentclass[11pt]{amsart}

\usepackage{amsfonts, amstext, amsmath}
\usepackage{graphicx, color, epsfig}

\newcommand{\vol}{\mathop{\rm vol}\nolimits}
\newcommand{\Tw}{\mathop{\rm tw}\nolimits}

\newcommand{\tild}[1]{{\widetilde{#1}}}

\newcommand{\HH}{{\mathbb{H}}}
\newcommand{\RR}{{\mathbb{R}}}
\newcommand{\ZZ}{{\mathbb{Z}}}
\newcommand{\CC}{{\mathbb{C}}}

\newcommand{\oo}{{\boldsymbol{o}}}
\newcommand{\pp}{{\boldsymbol{p}}}

\theoremstyle{plain}
\newtheorem{theorem}{Theorem}[section]

\newtheorem{lemma}[theorem]{Lemma}
\newtheorem{prop}[theorem]{Proposition}

\theoremstyle{definition}
\newtheorem{define}[theorem]{Definition}

\newtheorem*{namedtheorem}{\theoremname}
\newcommand{\theoremname}{testing}

\title[Hyperbolic geometry of multiply twisted knots]{Hyperbolic
geometry \\ of multiply twisted knots}

\author{Jessica S. Purcell}

\address{Jessica S. Purcell, Department of Mathematics, Brigham Young
University, Provo, UT 84602}

\email{jpurcell@math.byu.edu}

\begin{document}
\bibliographystyle{hamsplain}

\begin{abstract}
We investigate the geometry of hyperbolic knots and links whose
diagrams have a high amount of twisting of multiple strands.  We find
information on volume and certain isotopy classes of geodesics for the
complements of these links, based only on a diagram.  The results are
obtained by finding geometric information on generalized augmentations
of these links.
\end{abstract}

\maketitle


\section{Introduction}
\label{sec:intro}
By Mostow--Prasad rigidity and work of Gordon and Luecke
\cite{gordon-luecke}, the hyperbolic structure on the complement of a
hyperbolic knot is a knot invariant, and ought to be useful in
problems of knot and link classification.  In practice, this structure
seems difficult to compute.

In recent years, some geometric properties of hyperbolic knots and
links have been discovered for links admitting certain types of
diagrams, such as alternating links \cite{lackenby:alt-volume}, and
highly twisted knots and links \cite{purcell:cusps, purcell:volume,
fkp}.  However, many knots that are of interest to topologists and
hyperbolic geometers do not fall into these classes.  These include
Berge knots \cite{berge, baker:I, baker:II}, twisted torus knots and
Lorenz knots \cite{birman-kofman}, which contain many of the smallest
volume hyperbolic knots \cite{champanerkar-kofman}.  These knots often
have diagrams that are highly non-alternating, with few twists per
twist region, but contain regions where multiple strands of the
diagram twist around each other some number of times.  We would like
to be able to understand and estimate geometric properties of these
``multiply twisted'' knots and links, given only a diagram, but
currently we do not have the tools to do so.

In this paper, we take a first step toward such an understanding.  We
investigate the geometry of knots and links with diagrams with a high
amount of twisting of multiple strands.  We find information on the
geometry of these knots, including volume bounds and certain isotopy
classes of geodesics, based only on a diagram.  

The results are obtained \emph{augmenting} the knot and link diagrams.
That is, we encircle each twist of multiple strands by a simple closed
curve, unknotted in $S^3$.  The resulting link is called a generalized
augmented link, generalizing a construction of Adams in which two
twisting strands are encircled by an unknotted component
\cite{adams:aug}.  When one performs $1/n$ Dehn filling on the
augmentation components of these links, one adds $n$ full twists to
the strands.  All diagrams can be obtained by such twisting.  (See
section \ref{sec:character} for a more careful discussion.)  Hence
geometric information on a generalized augmented link, combined with
geometric information under Dehn filling, leads to geometric results
on knot complements.

Regular augmented links have a very nice hyperbolic structure,
including a decomposition into right angled ideal polyhedra, first
written down by Agol and Thurston
\cite[Appendix]{lackenby:alt-volume}.  Generalized augmented links do
not have as nice structure, but still contain enough symmetry to
obtain geometric estimates.  To obtain geometric information on Dehn
fillings of these manifolds, one may turn to results on cone
deformations due to Hodgson and Kerckhoff \cite{hk:rigid, hk:univ,
hk:shape}, or hyperbolike filling of Agol and Lackenby
\cite{agol:bounds, lackenby:word}, or volume change results due to
Futer, Kalfagianni, and the author \cite{fkp}.

We have investigated generalized augmented links elsewhere.  In
\cite{purcell:slopes}, we bounded the lengths of certain slopes on
these links, and showed that many knots obtained by their Dehn
fillings have meridian length approaching $4$ from below.  With Futer
and Kalfagianni, in \cite{fkp:coils} we investigated properties of
volumes of a very particular class of these links.  Here, we broaden
the results to larger classes of knots and links.

Finally, note that the focus of this paper is on geometric information
on hyperbolic generalized augmented links and their Dehn fillings.  In
a companion paper, we discuss results for generalized augmented links
which are not hyperbolic \cite{purcell:geom-aug}.

\subsection{Acknowledgements}

Research was partially funded by NSF grant DMS--0704359.  We thank
David Futer and John Luecke for helpful conversations.

\section{Characterization of generalized augmented links}
\label{sec:character}
We will be analyzing twisting and twist regions in a knot diagram.
Twist regions and generalized twist regions are defined carefully in
\cite{purcell:slopes}.  We review definitions here for convenience.

\begin{define}
	Let $K$ be a link in $S^3$, and let $D$ be a diagram of the link.
	We may view $D$ as a 4--valent graph with over--under crossing
	information at each vertex.  A \emph{twist region} of the diagram
	$D$ is a sequence of bigon regions of $D$ arranged end
	to end, which is maximal in the sense that there are no other bigons
	on either end of the sequence.  A single crossing adjacent to no
	bigons is also a twist region.

	We will assume throughout that the diagram is alternating within a
	twist region, else replace it with a diagram with fewer crossings in
	the twist region.
\label{def:twist}
\end{define}

In a \emph{twist region} of a diagram, two strands twist around each
other maximally, as in Figure \ref{fig:twist}(a), and bound a
``ribbon'' surface.

\begin{define}
A \emph{generalized twist region} of $D$ is a region of the diagram
where two or more strands twist around each other maximally, as in
Figure \ref{fig:twist}(b).  More precisely, a generalized twist region
is a region of the diagram consisting of $m\geq 2$ parallel strands.
When all the strands except the outermost two are removed from this
region of the diagram, the remaining two strands form a twist region.
In $S^3$, these two strands bound a ribbon surface between them.
Remaining strands of the generalized twist region can be isotoped to
lie parallel to each other, embedded on this ribbon surface.
\label{def:gen-twist-region}
\end{define}

\begin{figure}
	(a)
	\includegraphics{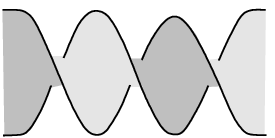}
	\hspace{.2in}
	(b)
	\includegraphics{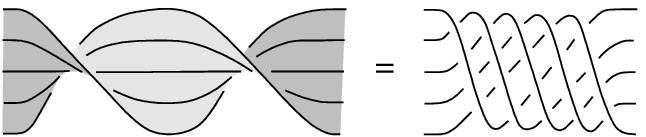}
	\caption{(a) A twist region.  (b) A generalized twist region.
	Multiple strands lie on the twisted ribbon surface.}
\label{fig:twist}
\end{figure}

The amount of twisting in each twist region is also important.  We
describe the amount of twisting in terms of half--twists and
full--twists.

\begin{define}
	Let $K$ be a link in $S^3$.  A \emph{half--twist} of a generalized
	twist region of a diagram consists of a single crossing of the two
	outermost strands.  The ribbon surface they bound, containing other
	strands of the twist region, flips over once in a half--twist.

	A \emph{full--twist} consists of two half--twists.  Figure
	\ref{fig:twist}(b) shows a single full--twist, or two half--twists,
	of five strands.
\label{def:half-twist}
\end{define}

Given a diagram of a link in $S^3$, group crossings into generalized
twist regions, such that each crossing is contained in exactly one
generalized twist region.  Call such a choice of generalized twist
regions a \emph{maximal twist region selection}.  Note the choice is
not necessarily unique.  For example, in Figure \ref{fig:twist}(b), we
could group the crossings shown into a single generalized twist region
containing a full--twist of five strands, or into twenty regular twist
regions, each containing a single half--twist of two strands.  Either
choice is a valid maximal twist region selection, although the former
seems more correct.

Now, at each generalized twist region in the maximal twist region
selection, insert a \emph{crossing circle}, that is, a simple closed
curve $C_i$ encircling the strands of the twist region, and bounding a
disk $D_i$ in $S^3$, perpendicular to the projection plane.  The $D_i$
are called \emph{twisting disks}.  See Figure \ref{fig:cross-cir}(a).
We can select the $C_i$ and the $D_i$ such that the collection of all
$D_i$ is a collection of disjoint disks in $S^3$.

When crossing circles are inserted at each twist region in the maximal
twist region selection, we obtain a new link, with components $K_j$
from the original link $K$, and crossing circles $C_i$.  The
complement of this link is homeomorphic to the complement of the link
$L$ obtained by untwisting at each $C_i$.  That is, we may remove all
full--twists from each generalized twist region of the link diagram
without changing the homeomorphism type of the link complement.  See
Figure \ref{fig:cross-cir}(b).

\begin{figure}
\begin{tabular}{ccccc}
	(a) & 
	\includegraphics{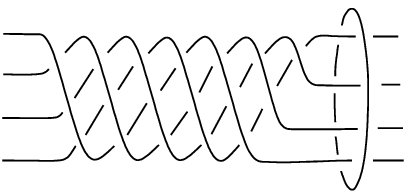} & \hspace{.2in} & (b) &
	\includegraphics{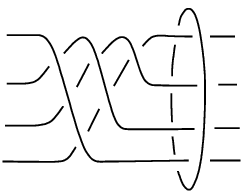}
\end{tabular}
\caption{(a) Encircle each twist region with a crossing circle.  (b)
	Link $L$ given by removing full--twists from the diagram.}
\label{fig:cross-cir}
\end{figure}

The resulting diagram of $L$ consists of unknotted link components
$C_i$ and components obtained from untwisting $K$, which we will call
$K_1, \dots, K_p$.  In the diagram of $L$, the components of $K$ will
either lie flat on the projection plane, or may have single
half--twists encircled by crossing circles.

\begin{define}
We call the link $L$ an \emph{augmentation} of the diagram $D$ of $K$,
or we say $L$ is the augmentation of the diagram $D$ corresponding to
a maximal twist region selection.  We also say that $L$ is obtained by
\emph{augmenting} $K$, and that $L$ is an \emph{generalized augmented}
link.
\label{def:augmentation}
\end{define}

For brevity, we often drop the adjective ``generalized'' from the term
generalized augmented links, since all augmented links we discuss here
are of this form.

The connection between $S^3-L$ and the original link complement is
given by Dehn filling.  Any slope $s$ on a torus $T^2$ is
parameterized by two relatively prime integers $p, q$, where $s =
p\mu+q\lambda$, and $\mu, \lambda$ generate $H_1(T^2; \ZZ)$.  When $M$
is the link complement $S^3-L$, at the $i$-th crossing circle $C_i$,
let $\mu_i, \lambda_i$ denote the meridian and longitude of $\partial
N(C_i)$, respectively.  Then Dehn filling along the slope $\mu_i +
n_i\lambda_i$ gives a new link whose diagram no longer contains $C_i$,
and the strands previously encircled by $C_i$ run through $n_i$
full--twists (see, for example, Rolfsen \cite{rolfsen}).  Thus Dehn
filling connects $S^3-K$ and the complement of the augmented link $L$.

\subsection{Reflection}

The link $L$ admits a reflection, as follows.  Arrange the diagram of
$L$ such that crossing circles of $L$ lie perpendicular to the
projection plane, and reflect the diagram of $L$ in the projection
plane.  The crossing circle components $C_i$ are taken to themselves.
Outside of twist regions, the diagram of $L$ is preserved.  If the
components $K_j$ lie flat on the projection plane, they are also
preserved by the reflection.

If some components $K_j$ run through a single half--twist at a twist
region, then the reflection will reverse all the crossings of the
half--twist, changing the direction of half--twist.  Apply a twist
homeomorphism, twisting one full twist at each half--twist in the
opposite direction.  This reverses the direction of the half--twist.
Thus the composition of the reflection and the twist homeomorphism is
an orientation reversing involution of $S^3-L$.

There is a surface which can be isotoped to be fixed pointwise by this
involution, namely, the projection plane outside of half--twists, and
the ribbon surfaces inside half twists, as well as a half--twisted
surface between $C_i$ and the knot strands.

The above discussion is a proof of the following, which is also
Proposition 3.1 of \cite{purcell:slopes}.

\begin{prop}
	Let $L$ be an augmentation of a diagram of a link in $S^3$.  Then
	$S^3-L$ admits a reflection, i.e. an orientation reversing
	involution with fixed point set a surface.
\label{prop:reflect}
\end{prop}

\section{Slopes lengths and hyperbolicity}

In this section, we prove results on slope lengths of generalized
augmented links.  Our methods generalize to hyperbolic manifolds which
admit a reflection, and we state the more general results.

\begin{lemma}
	Let $M$ be a $3$--manifold with torus boundary components with the
	following properties:
\begin{enumerate}
\item $M$ admits an orientation reversing involution $\sigma$ whose
	fixed point set is an embedded surface $P$ in $M$.
\item Some boundary component $T$ of $M$ meets $P$, and for some
	slope $\lambda$ on $T$, $\sigma$ is an orientation reversing
	involution of $\lambda$.  (Write $\sigma(\lambda) = -\lambda$.)
\end{enumerate}
Then $\lambda$ meets $P$ exactly twice.  
\label{lemma:lambda-P}  
\end{lemma}

When our manifold is in fact a generalized augmented link, $\lambda$
may be the slope $\partial D_i$ on $\partial N(C_i)$, for example, or
a slope $\partial D_i$ on $\partial N(K_j)$.

\begin{proof}
  Since $\sigma$ takes $\lambda$ to $-\lambda$, a representative
  of $\lambda$ (which, by abuse of notation, we will also call
  $\lambda$) has a fixed point under $\sigma$.  Thus $\lambda$
  meets $P$.  Additionally, since the only orientation reversing
  involutions of $S^1$ that fix a point must actually fix two points,
  $\lambda$ must meet $P$ twice.
\end{proof}

\begin{lemma}
	Let $M$ be as in Lemma \ref{lemma:lambda-P}.  Then the torus $T$
	is tiled by rectangles, each with one side parallel to the surface
	$P$, and one side orthogonal to $P$.  The lift of these rectangles
	to the universal cover $\widetilde{T}$ gives a lattice in $\RR^2$.
\label{lemma:lattice}
\end{lemma}

\begin{proof}
Consider the universal cover $\RR^2$ of the torus boundary component
$T$.  As $P$ is embedded, the slopes $P \cap T$ lift to give parallel
lines in $\RR^2$.  A simple curve representing the slope $\lambda$
lifts to give parallel lines perpendicular to the lines from $P$,
since $\lambda$ is taken to $-\lambda$ by the involution $\sigma$
fixing $P$.  The projection of these lines to $T$ gives a tiling of
$T$ by rectangles.  Together, the intersection points of these sets of
lines form a lattice $\ZZ^2$ of $\RR^2$.
\end{proof}

Construct a basis of the lattice of Lemma \ref{lemma:lattice} by
letting $\pp$ be a step parallel to a side from $P \cap T$, and by
letting $\oo$ be a step orthogonal to $\pp$.

\begin{lemma}
  Let $M$ be as in Lemma \ref{lemma:lambda-P}, and let $\{\pp, \oo\}$
  be the basis for the lattice on $\widetilde{T}$ as above.  Then the
  curve $\lambda$, which serves as one generator of $H_1(T;\ZZ)$, is
  given by $2\oo$.  Another generator of $H_1(T;\ZZ)$ is given by $\pp
  + \epsilon \,\oo$, where $\epsilon = 0$ if there are two components
  of $P\cap T$, and $\epsilon = 1$ if there is one component of $P\cap
  T$.
\label{lemma:generators}
\end{lemma}

\begin{proof}
By Lemma \ref{lemma:lambda-P}, $\lambda$ intersects $P$ twice.  Thus
its representative must cross lifts of $P$ twice in the lattice, and
be taken to itself under the involution in $P$, so it is $2\oo$.

Note this implies that all corners of the rectangles formed by $\pp$
and $\oo$ project to just two distinct points on $T$ under the
covering transformation.  These two points are the projection of $\oo$
and the projection of $2\oo$.  Additionally, the fact that
$\lambda=2\oo$ implies that $T$ is tiled by exactly two rectangles.  To
determine generators of $H_1(T;\ZZ)$, we determine if these
rectangles are glued with or without shearing on $T$.
  
Another obvious closed curve on $T$ besides $\lambda$ is given by a
single component of $P\cap \partial T$.  Call the corresponding slope
$\alpha$.  It does not necessarily generate $H_1(T;\ZZ)$ with
$\lambda$.  Since $\lambda$ intersects $P$ twice, either $\alpha$
intersects $\lambda$ once, in which case $P \cap T$ has two
components, there is no shearing, and $\pp$ is a generator; or
$\alpha$ intersects $\lambda$ twice, and $P\cap T$ has one component.

If $P\cap T$ has one component, then $\alpha=2\pp$, and $\alpha$ is
not a generator with $\lambda$.  Then $\pp$ must project to the same
point as $\oo$ under the covering projection, so $\pp + \oo$ will give
a closed curve on $T$.  Since it has intersection number $1$ with
$2\oo = \lambda$, $\pp+\oo$ will be a generator.
\end{proof}

When $M$ is known to admit a hyperbolic structure, we can find lower
bounds on the lengths of the arcs $\oo$ and $\pp$ in the lattice.
Recall that when a manifold has multiple cusps, lengths depend on a
choice of maximal cusps, i.e. a collection of disjoint horoball
neighborhoods, one for each cusp.  Lengths of arcs are measured on the
horospherical tori that form the boundaries of the horoball
neighborhoods.  To ensure lengths on a torus boundary are long, we
need to ensure that we can choose maximal cusps appropriately.

\begin{theorem}
Let $M$ be a 3--manifold with torus boundary components which admits a
complete finite volume hyperbolic structure, and has the following
additional properties:
\begin{enumerate}
\item $M$ admits an orientation reversing involution $\sigma$ whose
  fixed point set is an embedded surface $P$ in $M$.
\item Boundary components $T_1, \dots, T_t$ of $M$ meet $P$, and for
  each $T_i$, there is a slope $\lambda_i$ that is taken to
  $-\lambda_i$ under $\sigma$.
\end{enumerate}
Let $\{\pp_i, \oo_i\}$ generate the lattice on the universal cover
$\tild{T_i}$ of $T_i$, of intersections of lines which project to $P$
and lines which project orthogonal to $P$, respectively, as in Lemma
\ref{lemma:generators}.  Then there exists a choice of maximal cusps
of $M$ such that, when measured on these maximal cusps, the length of
each $\oo_i$ is at least $1$, and the length of $\pp_i$ is at least
$1/2$.
\label{thm:hyp-result}
\end{theorem}

Similar results were shown for particular classes of links in $S^3$ in
\cite{purcell:slopes}, using techniques of Adams \emph{et al.}
\cite{adams:II}.  We give a different proof here.

\begin{proof}
By Mostow--Prasad rigidity, the involution of $M$ is isotopic to an
isometry of $M$ under the hyperbolic metric.  The surface $P$, since
it is fixed pointwise, is isotopic to a totally geodesic surface in
$M$ (see for example \cite{menasco-reid}, \cite{leininger}).

Lift to the universal cover $\HH^3$, which we view as the upper half
space $\HH^3=\{(x,y,z)|z>0\}$.  For any $j$, we may conjugate such
that the cusp corresponding to $T_j$ lifts to the point at infinity.
The surface $P$ lifts to a collection of disjoint, totally geodesic
planes.

Since $P$ meets the cusp corresponding to $T_j$, copies of $P$ will
lift in $\HH^3$ to parallel vertical planes through infinity.  Because
$P$ is fixed under the involution $\sigma$, the collection of parallel
vertical planes must be preserved by a reflection of $\HH^3$ in any
one of the planes.  Hence the (Euclidean) distance between any two
adjacent planes must be constant.  Without loss of generality, we will
conjugate such that these vertical planes are the planes $y=n$, $n\in
\ZZ$, in $\HH^3 = \{(x,y,z)|z>0\}$, so that their Euclidean distance
is $1$.

The length of $\oo_j$ will be given by $1/c$, where $c$ is the
height of the horosphere bounding the horoball about infinity.  We
will show that we can always take $c$ to be less than or equal to $1$.

Define the horoball expansion about cusps of $T_1,
\dots, T_t$ such that the lengths of the $\oo_j$ agree for every $j$
simultaneously.  That is, there exists some (possibly large) $c$ such
that when each $\oo_j$ has length $1/c$, the horoballs about the cusps
corresponding to $T_1, \dots, T_t$ are disjoint.  Continue to increase
$c$ keeping all the $\oo_j$ of equal length, until the value $1/c$ is
as large as possible.  If there are remaining cusps disjoint from the
$T_j$, these may then be expanded in any way.

To prove the theorem, we must prove that the value of $c$ which
maximizes the length of the $\oo_j$ is less than or equal to $1$.

Suppose not.  Suppose $c>1$.  Since $c$ is minimal, horoballs about
cusps corresponding to some $T_i$ and $T_j$ must abut.  Conjugate such
that the cusp corresponding to $T_i$ is at infinity in $\HH^3$, with
lifts of $P$ corresponding to the planes $y=n, n\in\ZZ$.  The horoball
about infinity will have height $c$.  It will be tangent to some
horoball $H$ over a point $w$ on the boundary $\CC = \{(x,y,0)\}$ of
$\HH^3$, where $w$ projects to the cusp corresponding to $T_j$.  Since
$c>1$, note $H$ is a ball of Euclidean diameter $c>1$.

Because the diameter of $H$ is greater than $1$, $H$ must intersect a
plane $y=n$.  Because the reflection through the plane $y=n$ projects
to an isometry of $M$, the image of $H$ under this reflection must be
a horoball in $\HH^3$ disjoint from all other horoballs in the lift of
the maximal cusps.  Thus if $H$ lies over some point $w$ which is
\emph{not} on the plane $y=n$, then the image of $H$ under the
reflection through $y=n$ will give a horoball distinct from $H$, which
intersects $H$.  This is impossible.

So $H$ is centered at a point $w\in \CC$ which lies on a plane $y=n$.
Without loss of generality, assume $w=0$.  Thus we are assuming $0$
projects to some cusp corresponding to $T_j$ under the covering map.

Now consider $T_j$.  There is some isometry $S$ of $\HH^3$ taking $0$
to infinity and infinity to $0$, and taking lifts of $P$ which meet
the cusp $T_j$ to planes $y=n, n\in\ZZ$.  Note by the definition of
our horoball expansion, this isometry $S$ takes $H$ to a horoball of
height $z=c>1$ about infinity.

Consider $q = S^{-1}(i) = S^{-1}((0,1,0))$ on the boundary $\CC$ of
$\HH^3$.  This point lies on the boundary of some plane $Q$ of $\HH^3$
which projects to $P$ under the covering map.  This plane $Q$ is a
Euclidean hemisphere tangent to the plane $y=0$.  It has diameter at
most $1$, since it cannot intersect the plane $y=1$, which also
projects to $P$ under the covering map.

Consider the vertical geodesic in $\HH^3$ lying above $0$ in $\CC$.
There is a unique geodesic $\gamma$ from the point $q$ which meets
this vertical geodesic at a right angle.  The point $r$, where
$\gamma$ intersects the vertical geodesic, is of (Euclidean) height
$|q|$, where $|q|$ denotes the (Euclidean) distance of $q$ from $0$.
Because $q$ lies on the circle $Q$ of diameter at most $1$, $|q|$ is
at most $1$.  Because $H$ is of diameter $c>1$, $r$ must be contained
in $H$.  See Figure \ref{fig:r-in-H}.

\begin{figure}
  \input{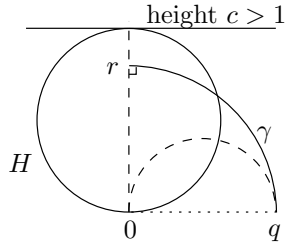}
  \caption{Note $r$ is contained in the horosphere $H$.}
\label{fig:r-in-H}
\end{figure}

But now consider the effect of the isometry $S$ on the geodesic
$\gamma$.  Since $S$ preserves the vertical geodesic above $0$ in
$\CC$, $S$ must take $\gamma$ to a geodesic from $S(q) = i \in \CC$ to
one meeting the vertical geodesic above $0$ at a right angle.  Thus
$S(r)$ will be of height exactly $1$.  On the other hand, $S(H)$ is of
height $c>1$, and $S(H)$ must contain $S(r)$.  This is impossible.

Thus all horoballs can be expanded to height $c\leq 1$.  It follows
that each $\oo_i$ has length at least $1$.

Finally, $\pp_i$ or $2\pp_i$ projects to a closed curve on $T_i$.
Hence translation along $\pp_i$ or $2\pp_i$ is a covering
transformation.  It must take a maximal horoball centered at a point
on $\CC$ to a disjoint maximal horoball.  Thus the translation length
is at least $1$, so $\pp_i$ has length at least $1/2$.
\end{proof}

We wish to study what happens when we twist along the disks $D_1,
\dots, D_t$, i.e. when we perform Dehn filling on slopes $1/n_1,
\dots, 1/n_t$ on the cusps corresponding to $C_1, \dots, C_t$,
respectively.  First, we give the following result about the lengths
of such slopes.  Note the following theorem applies to links in
general 3--manifolds, not just $S^3$.

\begin{prop}
	Let $L = C_1 \cup \dots C_t$ be a link in a $3$--manifold $M$, such
	that $M-L$ admits a complete, finite volume hyperbolic structure,
	admits an orientation reversing involution $\sigma$ whose fixed
	point set is a surface $P$, and for each component $C_i$ of $L$,
	there is a slope $\lambda_i$ taken to $-\lambda_i$ by $\sigma$.

	Let $\mu_i$ be the other generator of $H_1(\partial N(C_i))$ as in
	Lemma \ref{lemma:generators}.  Then the slope $\mu_i + n_i
	\,\lambda_i$
	has length at least $\sqrt{(1/4) + c_i^2}$.  Here:
\begin{enumerate}
	\item $c_i = 2|n_i|$ if $P \cap \partial N(C_i)$ consists of two
	curves, or
	\item $c_i = 2|n_i|-1$ if $P \cap \partial N(C_i)$ consists of one
	curve. 
\end{enumerate}
\label{prop:slope-lengths}
\end{prop}

\begin{proof}
$M-L$ fits the requirements of the lemmas above.  So in particular, by
Lemma \ref{lemma:generators}, $H_1(\partial N(C_i);\ZZ)$ is generated
by $2\oo_i$ and $\pp_i + \epsilon_i \,\oo_i$; the generator $2\oo_i$
corresponds to the curve $\lambda_i$; if $P \cap \partial N(C_i)$ has
two components, then one such component is a generator $\pp_i =
\mu_i$; and if $P \cap \partial N(C_i)$ has one component, then the
other generator is $\pp_i + \oo_i = \mu_i$.

Suppose first that $P \cap \partial N(C_i)$ has two components.  Then
the slope $\mu_i + n_i\, \lambda_i$ is given by $\pp_i + n_i\,
(2\,\oo_i)$.  Since $\pp_i$ and $\oo_i$ are orthogonal, by Theorem
\ref{thm:hyp-result} this slope has length at least
$\sqrt{(1/4) + 4\,n_i^2} = \sqrt{(1/4) + c_i^2}.$

Now suppose that $P \cap \partial N(C_i)$ has one component.  Then the
slope $\mu_i + n_i \,\lambda_i$ is given by $\pp_i + \oo_i +
n_i\,(2\,\oo_i) = \pp_i + (1 + 2\,n_i)\oo_i$.  It must have length at
least $\sqrt{(1/4) + (1-2|n_i|)^2} = \sqrt{(1/4) + c_i^2}.$
\end{proof}

\begin{define}
If $P \cap \partial N(C_i)$ consists of one curve, as in case (2) of
Proposition \ref{prop:slope-lengths}, we say there is a \emph{half--twist}
at $D_i$.
\label{def:half-twist-general}
\end{define}

This terminology comes from considering a neighborhood of $D_i$ in
$M$.  In this neighborhood, a half--twist at $D_i$ is identical to a
neighborhood of a half--twist of an augmented link in $S^3$, as in
Definition \ref{def:half-twist}.  See Figure \ref{fig:half-twist}.

Two half--twists in a row in a neighborhood of $D_i$ again yields a
full--twist in this neighborhood.  Thus Proposition
\ref{prop:slope-lengths} implies that the squared length of the slope
$\mu_i + n_i\lambda_i$ on $C_i$ is at least one more than the squared
number of half--twists inserted at $D_i$.

\begin{figure}
	\begin{center}
	\includegraphics{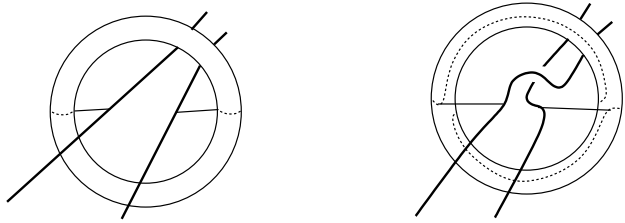}
	\end{center}
\caption{Left: $P\cap \partial N(C_i)$ has two components, shown in
	dotted lines.  Right: $P\cap \partial N(C_i)$ has one component,
	giving a half--twist. }
\label{fig:half-twist}
\end{figure}

\begin{theorem}
Let $K$ be a knot or link in $S^3$ which has a diagram $D$ and a
maximal twist region selection with at least $6$ half--twists in each
generalized twist region, and such that the corresponding augmentation
is hyperbolic.  Then $S^3-K$ is also hyperbolic.
\label{thm:hyp-knot}
\end{theorem}

\begin{proof}
The augmentation is a link with hyperbolic complement, by assumption.
It admits an orientation reversing involution $\sigma$ fixing a
surface $P$, and the cusps corresponding to crossing circles each have
a slope $\lambda_i$ which is taken to $-\lambda_i$ by $\sigma$:
namely, the slope of the longitude of the crossing circle.

The original knot or link complement is obtained from this link
complement by Dehn filling slopes on crossing circles.  The longitude
of a crossing circle is given by $\lambda_i$.  The meridian is the
generator $\mu_i$ of Proposition \ref{prop:slope-lengths}.  If the
knot has $c_i$ half twists in the $i$-th twist region, then the Dehn
filling slope is $\mu_i + n_i \lambda_i$, where $n_i = c_i/2$ if $c_i$
is even, $n_i = (c_i+1)/2$ if $c_i$ is odd.

By Proposition \ref{prop:slope-lengths}, the slope of the Dehn filling
has length at least $\sqrt{(1/4)+c_i^2} > 6$, since the diagram of $K$
has at least $6$ half--twists in each generalized twist region.  Thus
by the $6$--Theorem (\cite{agol:bounds}, \cite{lackenby:word}), the
manifold resulting from Dehn filling is hyperbolic.
\end{proof}

\section{Volumes}

The existence of a reflection gives information about the volumes of
augmented links as well.  Theorem \ref{thm:hyp-vol}, below, is an
immediate generalization of a similar theorem in \cite{fkp}.

\begin{lemma}
Let $K$ be a knot or link in $S^3$ which has a diagram $D$ and a
maximal twist region selection such that the corresponding
augmentation yields a link $L$ in $S^3$ whose complement is
hyperbolic.  Then the volume satisfies
$$ \vol(S^3-L) \geq 2\,v_8\,(\Tw(D)-1), $$
where $v_8 \approx 3.66386$ is the volume of a regular hyperbolic
octahedron, and $\Tw(D)$ is the number of generalized twist regions of
the maximal twist region selection of $D$.
\label{lemma:vol-refl}
\end{lemma}

\begin{proof}
By assumption, $S^3-L$ admits a complete hyperbolic structure.  By
Proposition \ref{prop:reflect}, it admits a reflective symmetry.  Thus
$S^3-L$ contains a surface $P$ fixed pointwise under the reflection.

Cut $S^3-L$ along this surface. This produces a (possibly
disconnected) manifold $N$ with totally geodesic boundary.  By a
theorem of Miyamoto \cite{miyamoto}, the volume of $N$ is at least
$-v_8\,\chi(N)$, where $\chi(N)$ denotes the Euler characteristic of
$N$.

Now, in the case that $P$ is the projection plane (i.e. no
half--twists), cutting along $P$ splits $S^3-L$ into two balls, with
half arcs corresponding to crossing circles drilled out of the ball.
This is a handlebody.  Since there are $\Tw(D)$ crossing circles, the
genus of the handlebody is $\Tw(D)$.  Thus we obtain the volume
estimate:
$$ \vol(S^3-L) \geq 2\,v_8\,(\Tw(D)-1). $$

When the diagram has half--twists, let $L'$ denote the link obtained
by removing all half--twists from the diagram of $L$.  Topologically,
$S^3-L'$ is obtained from $S^3-L$ by cutting $S^3-L$ along the disks
bounded by crossing circles, and regluing with a half--twist.

Note $S^3-L'$ has the following description as a gluing of ideal
polyhedra.  Cut $S^3-L'$ along the projection plane.  This slices each
of the disks bounded by crossing circles in half.  Now cut along each
of these half disks and pull the disks apart.  See Figure
\ref{fig:top-decomp}.

\begin{figure}
\includegraphics{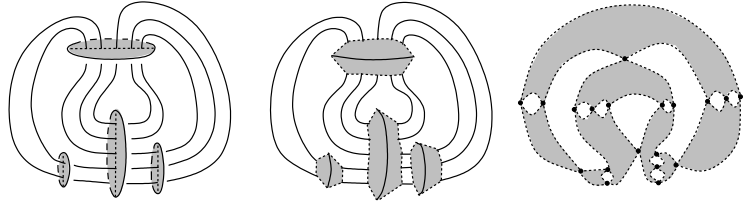}
\caption{Decomposing $S^3-L'$ into ideal polyhedra. First, cut along
	$P$.  Second, cut along half disks.  Finally, shrink remaining link
	components to ideal vertices.}
\label{fig:top-decomp}
\end{figure}

This separates $S^3-L'$ into two identical ideal polyhedra with faces
given by crossing disks and by the projection plane.  We may glue
these polyhedra back in the manner in which we cut them to obtain
$S^3-L'$.  We may also change the gluing on crossing disks only to
obtain $S^3-L$, as follows.  Rather than glue crossing disks straight
across where $L$ has a half--twist, glue a half crossing disk on one
polyhedron to the opposite half crossing disk on the opposite
polyhedron, inserting the half--twist.  See Figure
\ref{fig:poly-half-twist}.

\begin{figure}
\input{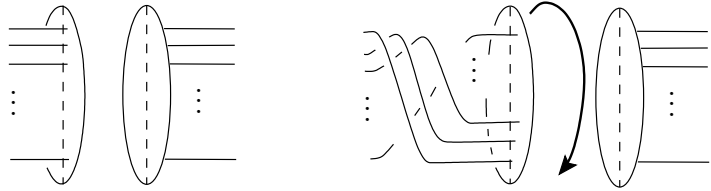}
\caption{Left: Gluing without a half twist.  Right: Inserting a half--twist.}
\label{fig:poly-half-twist}
\end{figure}

Compute the Euler characteristic of the cut manifold $(S^3-L)-P$ by
reading it off this polyhedral decomposition.  Since $(S^3-L)-P$ has
boundary, it retracts onto a one--skeleton.  Build the one--skeleton
by including a vertex for each ideal polyhedron (two vertices).  Edges
run through the half crossing disks which we glue.  There will be one
edge per glued pair of half crossing disks.  Since there are $\Tw(D)$
crossing disks, the Euler characteristic is $2 - 2\Tw(D)$.  Thus by
Miyamoto's theorem, the volume satisfies: $\vol(S^3-L) \geq
2\,v_8\,(\Tw(D)-1).$
\end{proof}

Lemma \ref{lemma:vol-refl} should be compared to Proposition 3.1 of
\cite{fkp}.  The proof above is an immediate extension of the proof of
that theorem to this more general case.  For links with two strands
per twist region, we showed in \cite{fkp} that Lemma
\ref{lemma:vol-refl} is sharp.

In general, when crossing circles have more than two strands per twist
region, Lemma \ref{lemma:vol-refl} seems to actually be far from
sharp.  With Futer and Kalfagianni we have been able to develop better
bounds on volumes of a certain class of knots \cite{fkp:coils}.
Meanwhile, Lemma \ref{lemma:vol-refl} gives a working lower bound on
volumes.

\begin{theorem}
Let $K$ be a knot or link in $S^3$ which has a diagram $D$ and a
maximal twist region selection with at least $7$ half--twists in each
generalized twist region, and such that the corresponding augmentation
is hyperbolic. Let $\Tw(D)$ denote the number of generalized twist
regions in the maximal twist region selection.  Then \[ \vol(S^3-K)
\geq 0.64756\,(\Tw(D) -1). \]
\label{thm:hyp-vol}
\end{theorem}

\begin{proof}
Let $L$ be the augmentation, $S^3-L$ hyperbolic, by assumption.  By
Lemma \ref{lemma:vol-refl}, the volume satisfies:
$$\vol(S^3-L) \geq 2\,v_8\,(\Tw(D)-1).$$

Now, $S^3-K$ is obtained by Dehn filling $S^3-L$.  Since there are at
least $7$ half--twists per twist region, by Proposition
\ref{prop:slope-lengths}, the Dehn filling is along slopes of length
at least $\sqrt{49.25} > 2\pi$.  Apply Theorem 1.1 of \cite{fkp}.
This theorem states that if $M$ is a hyperbolic manifold, and $s_1,
\dots, s_k$ are slopes on cusps of $M$ with minimum length
$\ell_{min}$ at least $2\pi$, then the Dehn filled manifold $M(s_1,
\dots, s_k)$ is hyperbolic with volume bounded below by
$$ \vol(M(s_1,\dots, s_k)) \geq
\left(1-\left(\frac{2\pi}{\ell_{min}}\right)^2\right)^{3/2}
\,\vol(M). $$

In our case, $\ell_{min} \geq \sqrt{49.25}$ and the volume of the
unfilled manifold $S^3-L$ satisfies $\vol(S^3-L) \geq
2\,v_8\,(\Tw(D)-1)$.  Thus the volume of $S^3-K$ satisfies
\begin{eqnarray*}
	\vol(S^3-K) &\geq&
\left(1-\left(\frac{2\pi}{\sqrt{49.25}}\right)^2\right)^{3/2}
\,2\,v_8\,(\Tw(D)-1)\\ &> & 0.64756\,(\Tw(D)-1).
\end{eqnarray*}
\end{proof}

\section{Geodesics}

We now give information on classes of geodesics in knot complements.
Our tools are those of cone manifolds and cone deformations.  We
briefly review the definitions and results we use.

\begin{define}
	A \emph{hyperbolic cone manifold} is a $3$--manifold $M$ and a link
	$\Sigma$ in $M$ such that $M-\Sigma$ admits an incomplete hyperbolic
	metric, with cone singularities along $\Sigma$.  That is, a
	neighborhood of $\Sigma$ in $M$ has a metric whose cross section is
	a hyperbolic cone, with cone angle $\alpha$ at the core.

	A \emph{hyperbolic cone deformation} is a one--parameter family of
	hyperbolic cone manifold structures on $M-\Sigma$.  
\end{define}

In special cases, a Dehn filling can be realized geometrically as a
cone deformation, as follows.  Suppose $M$ is a $3$--manifold with
torus boundary which admits a complete hyperbolic metric.  Let $s$ be
a slope on $\partial M$.  Then we may view the complete hyperbolic
structure on $M$ as a hyperbolic cone manifold structure on $M(s)$
with cone angle zero along the link at the core of the attached solid
torus in $M(s)$.

We may always increase the cone angle from $\alpha=0$ to
$\alpha=\varepsilon$, for some $\varepsilon>0$ via cone deformation,
by work of Hodgson and Kerckhoff \cite{hk:rigid}.
When $\alpha=\varepsilon$, in the hyperbolic cone metric, the slope $s$
will bound a singular disk.  That is, a representative of $s$ can be
isotoped to bound a disk $D$ which admits a smooth hyperbolic metric
everywhere except at the core of $D$, where $D$ intersects the
singular locus $\Sigma$.  Thus this manifold with the hyperbolic cone
metric is homeomorphic to $M(s)$.

In case there is a cone deformation starting at cone angle $\alpha=0$
and extending to $\alpha=2\pi$, the final structure when $\alpha=2\pi$
gives a complete, non-singular hyperbolic metric on the manifold
$M(s)$.  In this case, we say the Dehn filling is \emph{realized by
cone deformation}.

The benefit of a cone deformation is that one obtains some geometric
control on the hyperbolic structure of the manifold.  In particular,
when we have a single filling slope, the core of the Dehn filled solid
torus is a closed geodesic in the hyperbolic structure given by cone
angle $\alpha=2\pi$.  Thus this core is isotopic to a geodesic
provided we can show a Dehn filling is realized by cone deformation.

Hodgson and Kerckhoff analyzed conditions which guarantee the
existence of a cone deformation \cite{hk:univ}.  We will apply their
results, but first we need the following definition.

\begin{define}
	Let $M$ be a $3$--manifold with torus boundary $\partial M = T$
	admitting a complete hyperbolic metric.  Let $s$ be a slope on $T$.
	In the hyperbolic structure on $M$, $T$ becomes a cusp.  Take any
	embedded horoball neighborhood of this cusp and consider its
	boundary.  This inherits a Euclidean metric from the hyperbolic
	structure on $M$.  Thus we may measure the length of $s$ and the
	area of the Euclidean torus $T$ with respect to this metric.

	Define the \emph{normalized length} of $s$ to be
	\[ \ell_{norm}(s) = \frac{\rm{length}(s)}{\sqrt{\rm{area(T)}}}. \]
	Here $\rm{length}(s)$ the length of a geodesic representing $s$.
	Note that unlike the lengths of Theorem \ref{thm:hyp-result}, the
	normalized length of a slope is independent of choice of horoball
	neighborhood about the cusp corresponding to $T$.
\end{define}

The following is a consequence of Theorem 1.2 of \cite{hk:shape}.

\begin{theorem}[(Hodgson--Kerckhoff)]
Consider a complete, finite volume hyperbolic structure on the
interior of a compact, orientable 3--manifold $M$ with $k\geq 1$ torus
boundary components.  Let $T_1, \dots, T_k$ be horospherical tori
which are embedded as cross--sections to the cusps of the complete
structure.  Let $s_1, \dots, s_k$ be slopes, $s_i$ on $T_i$.  Then
$M(s_1, \dots, s_k)$ admits a complete hyperbolic structure in which
the core cures of the Dehn filled solid tori are isotopic to
geodesics, provided the normalized lengths $\hat{L_i} =
\ell_{norm}(s_i)$ satisfy
\[ \sum_{i=1}^k \frac{1}{\hat{L_i}^2} < \frac{1}{(7.5832)^2}. \]
\label{thm:hk}
\end{theorem}

Theorem 1.2 of \cite{hk:shape} is actually a more general theorem
about Dehn filling space for manifolds with multiple cusps.  However,
in the proof of that theorem it is shown that under the above
assumptions on normalized lengths of slopes, a cone deformation exists
from cone angle $0$ to $2\pi$ for which each component of the singular
locus has a tube about it of radius at least
$\rm{arctanh}(1/\sqrt{3})$ (page 36 of \cite{hk:shape}).  The
components of the singular locus correspond to the cores of the filled
solid tori.  Since each has a tube about it throughout the
deformation, the cores remain isotopic to geodesics.  See also the
explanation in \cite{hk:shape} on page 5, after the statement of
Theorem 1.2.

\begin{lemma}
Let $M$, $L$, $\lambda_i$, and $\mu_i$ be as in Proposition
\ref{prop:slope-lengths}.  Then the normalized length of each slope
$s_i = \mu_i + n_i\,\lambda_i$ is at least
\[ \ell_{norm}(s_i) \geq \sqrt{c_i}, \]
where again $c_i$ is the number of half--twists inserted by the Dehn
filling along slope $s_i$.
\label{lemma:norm-lengths}
\end{lemma}

The proof of Lemma \ref{lemma:norm-lengths} is similar to that of
Proposition \ref{prop:slope-lengths}, except with the added difficulty
that we are considering normalized lengths, and not actual lengths.
Compare to \cite[Proposition 6.5]{purcell:cusps}.  

\begin{proof}
Write the slope $s_i=\mu_i + n_i\,\lambda_i$ in terms of the lengths
of $\oo_i$ and $\pp_i$, of Lemma \ref{lemma:generators}.  In
particular, as in Proposition \ref{prop:slope-lengths}, the slope is
given by $\pp_i + c_i\,\oo_i$, where $c_i$ is the number of
half--twists inserted by the Dehn filling, and since $\oo_i$ and
$\pp_i$ are orthogonal, its length is given by $\sqrt{p_i^2 +
c_i^2\,o_i^2}$, where $p_i$ and $o_i$ denote the lengths of geodesic
representatives of $\pp_i$ and $\oo_i$.  By Lemma
\ref{lemma:generators}, the area of the cusp torus is given by
$2o_ip_i$.

Thus the normalized length of $s_i= \mu_i + n_i\,\lambda_i$ is given
by
\[ \ell_{norm}(s_i) =
\frac{\sqrt{p_i^2 + c_i^2o_i^2}}{\sqrt{2p_io_i}} =
\sqrt{\frac{p_i}{2o_i} + \frac{c_i\,o_i}{2p_i}}. \]

Minimize the normalized length with respect to $p_i/o_i$.  We
find that its value is minimum when the ratio $p_i/o_i$ equals
$c_i$.  In this case, the normalized length will be $\sqrt{c_i}$.
\end{proof}

We may now prove Theorem \ref{thm:geodesics}, giving results on
isotopy classes of geodesics in generalized augmented links.

\begin{theorem}
	Let $K$ be a knot or link in $S^3$ which has a diagram $D$ and a
maximal twist region selection with $\Tw(D)$ twist regions, such that
the corresponding augmentation is hyperbolic.  Let $c_i$ be the number
of half--twists in the $i$-th twist region.  Then each crossing circle
is isotopic to a geodesic in the hyperbolic structure on $S^3-K$,
provided
\[ \sum_{i=1}^{\Tw(D)} \frac{1}{c_i} < \frac{1}{(7.5832)^2}.\]
\label{thm:geodesics}
\end{theorem}

\begin{proof}
$S^3-K$ is obtained from $S^3-L$ by Dehn filling the crossing circles.
By Lemma \ref{lemma:norm-lengths}, the normalized lengths of the
slopes of the Dehn filling are at least $\sqrt{c_i}$, where $c_i$ is
the number of half--twists in the $i$-th generalized twist region of
$D$.  By Theorem \ref{thm:hk}, the cores of the filled solid tori are
isotopic to geodesics provided
\begin{eqnarray*}
\sum_{i=1}^{\Tw(D)} \frac{1}{c_i} &<& \frac{1}{(7.5832)^2}.
\end{eqnarray*}
\end{proof}

\bibliography{references}

\providecommand{\bysame}{\leavevmode\hbox to3em{\hrulefill}\thinspace}
\providecommand{\href}[2]{#2}
\begin{thebibliography}{10}

\bibitem{adams:II}
Colin Adams, Hanna Bennett, Christopher Davis, Michael Jennings, Jennifer
  Kloke, Nicholas Perry, and Eric Schoenfeld, \emph{Totally geodesic {S}eifert
  surfaces in hyperbolic knot and link complements. {II}}, J. Differential
  Geom. \textbf{79} (2008), no.~1, 1--23.

\bibitem{adams:aug}
Colin~C. Adams, \emph{Augmented alternating link complements are hyperbolic},
  Low-dimensional topology and Kleinian groups (Coventry/Durham, 1984), London
  Math. Soc. Lecture Note Ser., vol. 112, Cambridge Univ. Press, Cambridge,
  1986, pp.~115--130.

\bibitem{agol:bounds}
Ian Agol, \emph{Bounds on exceptional {D}ehn filling}, Geom. Topol. \textbf{4}
  (2000), 431--449 (electronic).

\bibitem{baker:I}
Kenneth~L. Baker, \emph{Surgery descriptions and volumes of {B}erge knots. {I}.
  {L}arge volume {B}erge knots}, J. Knot Theory Ramifications \textbf{17}
  (2008), no.~9, 1077--1097.

\bibitem{baker:II}
\bysame, \emph{Surgery descriptions and volumes of {B}erge knots. {II}.
  {D}escriptions on the minimally twisted five chain link}, J. Knot Theory
  Ramifications \textbf{17} (2008), no.~9, 1099--1120.

\bibitem{berge}
John Berge, \emph{Some knots with surgery yielding lens spaces}, unpublished
  manuscript.

\bibitem{birman-kofman}
Joan Birman and Ilya Kofman, \emph{A new twist on {L}orenz links},
  \mbox{arXiv:0707.4331}.

\bibitem{champanerkar-kofman}
Abhijit Champanerkar, Ilya Kofman, and Eric Patterson, \emph{The next simplest
  hyperbolic knots}, J. Knot Theory Ramifications \textbf{13} (2004), no.~7,
  965--987.

\bibitem{fkp:coils}
David Futer, Efstratia Kalfagianni, and Jessica~S. Purcell, \emph{On
  diagrammatic bounds of knot volumes and spectral invariants},
  \mbox{arXiv:math/0901.0119}.

\bibitem{fkp}
\bysame, \emph{Dehn filling, volume, and the {J}ones polynomial}, J.
  Differential Geom. \textbf{78} (2008), no.~3, 429--464.

\bibitem{gordon-luecke}
Cameron~McA. Gordon and John Luecke, \emph{Knots are determined by their
  complements}, J. Amer. Math. Soc. \textbf{2} (1989), no.~2, 371--415.

\bibitem{hk:rigid}
Craig~D. Hodgson and Steven~P. Kerckhoff, \emph{Rigidity of hyperbolic
  cone-manifolds and hyperbolic {D}ehn surgery}, J. Differential Geom.
  \textbf{48} (1998), no.~1, 1--59.

\bibitem{hk:univ}
\bysame, \emph{Universal bounds for hyperbolic {D}ehn surgery}, Ann. of Math.
  (2) \textbf{162} (2005), no.~1, 367--421.

\bibitem{hk:shape}
\bysame, \emph{The shape of hyperbolic {D}ehn surgery space}, Geom. Topol.
  \textbf{12} (2008), no.~2, 1033--1090.

\bibitem{lackenby:word}
Marc Lackenby, \emph{Word hyperbolic {D}ehn surgery}, Invent. Math.
  \textbf{140} (2000), no.~2, 243--282.

\bibitem{lackenby:alt-volume}
\bysame, \emph{The volume of hyperbolic alternating link complements}, Proc.
  London Math. Soc. (3) \textbf{88} (2004), no.~1, 204--224, With an appendix
  by Ian Agol and Dylan Thurston.

\bibitem{leininger}
Christopher~J. Leininger, \emph{Small curvature surfaces in hyperbolic
  3-manifolds}, J. Knot Theory Ramifications \textbf{15} (2006), no.~3,
  379--411.

\bibitem{menasco-reid}
William Menasco and Alan~W. Reid, \emph{Totally geodesic surfaces in hyperbolic
  link complements}, Topology '90 (Columbus, OH, 1990), Ohio State Univ. Math.
  Res. Inst. Publ., vol.~1, de Gruyter, Berlin, 1992, pp.~215--226.

\bibitem{miyamoto}
Yosuke Miyamoto, \emph{Volumes of hyperbolic manifolds with geodesic boundary},
  Topology \textbf{33} (1994), no.~4, 613--629.

\bibitem{purcell:geom-aug}
Jessica~S. Purcell, \emph{On multiply twisted knots that are {S}eifert fibered
  or toroidal}, \mbox{arXiv:0906.4575}.

\bibitem{purcell:volume}
\bysame, \emph{Volumes of highly twisted knots and links}, Algebr. Geom. Topol.
  \textbf{7} (2007), 93--108.

\bibitem{purcell:cusps}
\bysame, \emph{Cusp shapes under cone deformation}, J. Differential Geom.
  \textbf{80} (2008), no.~3, 453--500.

\bibitem{purcell:slopes}
\bysame, \emph{Slope lengths and generalized augmented links}, Comm. Anal.
  Geom. \textbf{16} (2008), no.~4, 883--905.

\bibitem{rolfsen}
Dale Rolfsen, \emph{Knots and links}, Publish or Perish Inc., Berkeley, Calif.,
  1976, Mathematics Lecture Series, No. 7.

\end{thebibliography}

\end{document}